\documentclass[12pt,twoside]{amsart}
\usepackage{amssymb}
\usepackage{verbatim}
\usepackage{amsmath}
\usepackage{bm}
\usepackage{a4wide}
\usepackage[utf8]{inputenc}
\usepackage[T1]{fontenc}
\usepackage{times}
\usepackage{amssymb,latexsym}
\usepackage{enumerate}
\usepackage[stable]{footmisc}
\usepackage{color}
\usepackage[colorlinks,linkcolor=black,citecolor=black,urlcolor=black]{hyperref}

\makeatletter
\newcommand{\sumprime}{\if@display\sideset{}{'}\sum%
            \else\sum'\fi}
\makeatother

\begin{document}

\numberwithin{equation}{section}

% define theorem environments
\newtheorem{theorem}{Theorem}[section]
\newtheorem{proposition}[theorem]{Proposition}
\newtheorem{conjecture}[theorem]{Conjecture}
\def\theconjecture{\unskip}
\newtheorem{corollary}[theorem]{Corollary}
\newtheorem{lemma}[theorem]{Lemma}
\newtheorem{observation}[theorem]{Observation}
\newtheorem{definition}{Definition}
\numberwithin{definition}{section} %\def\thedefinition{\unskip}
\newtheorem{remark}{Remark}
\def\theremark{\unskip}
\newtheorem{kl}{Key Lemma}
\def\thekl{\unskip}
\newtheorem{question}{Question}
\def\theexample{\unskip}
\newtheorem{problem}{Problem}
\newtheorem{example}{Example}

\thanks{The first author thanks the Centre Henri Lebesgue ANR-11-LABX-0020-01 for creating an attractive mathematical environment. The second and third authors are supported by National Natural Science Foundation of China, No. 12271101.}

\address[Gilles Carron]{Nantes Université, CNRS, Laboratoire de Mathématiques Jean Leray, LMJL, UMR 6629, F-44000 Nantes, France}
\email{gilles.carron@univ-nantes.fr}

\address [Bo-Yong Chen] {School of Mathematical Sciences, Fudan University, Shanghai, 200433, China}
\email{boychen@fudan.edu.cn}

\address [Yuanpu Xiong] {School of Mathematical Sciences, Fudan University, Shanghai, 200433, China}
\email{ypxiong@fudan.edu.cn}

\title{Type problem, the first eigenvalue and Hardy inequalities}
\author{Gilles Carron,  Bo-Yong Chen and Yuanpu Xiong}

\date{}

\begin{abstract}

In this paper,  we study the relationship between the type problem and the asymptotic behaviour of the first (Dirichlet) eigenvalues $\lambda_1(B_r)$ of ``balls'' $B_r:=\{\rho<r\}$ on a complete Riemannian manifold $M$ as $r\rightarrow +\infty$, where $\rho$ is a Lipschitz continuous exhaustion function with $|\nabla\rho|\leq1$ a.e. on $M$.  We obtain several sharp results. First, if for all $r>r_0$
\[
 r^2 \lambda_1(B_r)\ge \gamma>0,
 \]
 we obtain a sharp estimate of the volume growth: $|B_r|\ge cr^{\mu(\gamma)}.$ Moreover when $\gamma>j_0^2\approx 5.784$,
where $j_0$ denotes the first positive zero of the Bessel function $J_0$, then $M$ is hyperbolic and we have a Hardy type inequality. In the case where $r_0=0$, a sharp Hardy type inequality holds. These spectral conditions are satisfied if one assumes that $\Delta\rho^2\geq2\mu(\gamma)>0$. In particular, when $\inf_M\Delta\rho^2>4$, $M$ is hyperbolic and we get a sharp Hardy type inequality. Related results for finite volume case are also studied.
 
\end{abstract}

\maketitle

\tableofcontents

\section{Introduction}

Let $(M,g)$ be a complete,  non-compact Riemannian manifold with $\dim{M}\geq2$,  and denote by $\Delta$ the Laplace operator associated to $g$. An upper semicontinuous function $u$ on $M$ is called {\it subharmonic} if $\Delta u\ge 0$ holds in the sense of distributions.  If every negative subharmonic function on $M$ has to be a constant,  then $M$ is said to be {\it parabolic};  otherwise $M$ is called {\it hyperbolic}.  It is well-known that $M$ is parabolic (resp.  hyperbolic) if and only if the Green function $G_M(x,y)$ is infinite (resp.  finite) for all $x\neq{y}$;  or the Brownian motion on $M$  is recurrent (resp.  transient).

The type problem is how to decide the parabolicity and   hyperbolicity through intrinsic geometric conditions.  The case of surfaces is classical,  for  the type of $M$  depends only on the conformal class of $g$,  i.e.,  the complex structure determined by $g$.  Ahlfors \cite{Ahlfors35} and Nevanlinna \cite{Nevanlinna40} first showed that $M$ is parabolic whenever
\begin{equation}\label{eq:parabolic_volume_2}
\int^{+\infty}_1\frac{dr}{|\partial{B(x_0,r)}|}=+\infty,
\end{equation}
where $B(x_0,r)$ is the geodesic ball with center  $x_0\in{M}$ and radius $r$. 
The same conclusion was extended to  high dimensional cases by  Lyons-Sullivan \cite{LyonsSullivan} and Grigor'yan \cite{Grigoryan83,Grigoryan85}.   Moreover,   \eqref{eq:parabolic_volume_1} can be relaxed to
\begin{equation}\label{eq:parabolic_volume_1}
\int^{+\infty}_1 \frac{rdr}{|B(x_0,r)|}=+\infty
\end{equation}
 (cf. Karp \cite{Karp82}, Varopolous \cite{Varopoulos} and Grigor'yan \cite{Grigoryan83,Grigoryan85}, see also Cheng-Yau \cite{ChengYau75}).  We refer to the excellent survey \cite{Grigoryan} of Grigor'yan for other sufficient conditions of parabolicity.  
 
On the other side,  it seems more difficult to find sufficient conditions for hyperbolicity.  Yet there is a classical result  stating that $M$ is hyperbolic whenever  the first (Dirichlet) eigenvalue $\lambda_1(M)$ of $M$ is positive. Recall that 
$$
\lambda_1(M): = \lim_{j\rightarrow +\infty} \lambda_1(\Omega_j)
$$
for some/any increasing sequence of precompact open sets $\{\Omega_j\}$ in $M$,  such that $M=\bigcup \Omega_j$.  Here given an open set $\Omega\subset\subset M$,   define 
$$
\lambda_1(\Omega):=\inf\left\{\frac{\int_\Omega |\nabla \phi|^2 dV}{\int_\Omega \phi^2 dV}:\phi\in \mathrm{Lip}_0(M),\ \mathrm{supp}\,\phi\subset\overline{\Omega},\ \phi\not\equiv 0 \right\},
$$
where $\mathrm{Lip}_0(M)$ denotes the set of Lipschitz continuous functions on $M$ with compact supports. 

\begin{remark}
We use the convention that $\lambda_1(\varnothing)=+\infty$.
\end{remark}

The main focus of this paper is to determine the hyperbolicity in the case  $\lambda_1(M)=0$. Grigor'yan showed that $M$ is hyperbolic if the following Faber-Krahn type inequality holds:
$$
\lambda_1(\Omega) \ge f(|\Omega|),\ \ \ \forall\,\Omega\subset\subset M: |\Omega|\ge v_0>0,
$$
where $f$ is a positive decreasing function on $(0,+\infty)$ such that $\int^{+\infty}_{v_0} \frac{dv}{v^2 f(v)}<+\infty$ (see, e.g., \cite{Grigoryan}, Theorem 10.3). We shall use certain quantity measuring the asymptotic behaviour of $\lambda_1(B_r)$ for certain ``balls'' $B_r$ as $r\rightarrow +\infty$, which seems to be easier to analyse. More precisely, let us first fix  a nonnegative locally Lipschitz continuous function $\rho$ on $M$, which is an exhaustion function (i.e., $B_r:=\{\rho<r\}\subset\subset{M}$ for any $r>0$), such that $|\nabla \rho|\le 1$ holds a.e.  on $M$. Note that  if $\rho$ is the distance ${\rm dist}_M(x_0,\cdot)$ from some $x_0\in M$,  then $B_r$ is precisely the geodesic ball $B(x_0,r)$.

To state the main result, we denote by $\lambda_\mu$ the first eigenvalue of the Laplace operator on $[0,1)$ for the Dirichlet condition at $s=1$ and with respect to the measure $s^{\mu-1}ds$, i.e.,
\begin{equation}\label{eq:lambda_mu}
\lambda_\mu=\inf\left\{ \frac{\int_0^1 \left|\psi'(s)\right|^2 s^{\mu-1}ds}{\int_0^1 \left|\psi(s)\right|^2 s^{\mu-1}ds}:  \ \psi\in \mathrm{Lip}([0,1]),\ \psi(1)=0,\ \psi\not\equiv0\right\}.
\end{equation}
It is known that $\lambda_\mu=j^2_{\mu/2-1}$,  where $j_\nu$ is the first positive zero of the Bessel function $J_\nu$,  with the infimum $\lambda_\mu$ realized by
\begin{equation}\label{eq:psi_mu}
\psi_\mu(s):=s^{1-\mu/2}J_{\mu/2-1}\left(\sqrt{\lambda_\mu} s\right).
\end{equation}
We have

\begin{theorem}\label{th:main}
Suppose that
\begin{equation}\label{eq:eigenvalue_quadratic_decay}
\lambda_1(B_r)\geq\frac{\lambda_\mu}{r^2},
\end{equation}
holds for all $r\geq{r_0}$. Then the following properties hold:

\begin{itemize}
\item[$(1)$]
There is a constant $c>0$ such that 
\begin{equation}\label{eq:sharp_volume_growth}
|B_r|\geq {c}r^\mu,\ \ \ \forall\,r\geq{r_0}.
\end{equation}
Here, the constant $c$ might depend on the geometry of $B_{r_0}$.

\item[$(2)$]
If $\mu>2$, then $M$ is hyperbolic and the Hardy type inequality
\begin{equation}\label{eq:HardyIneq}
C\int_M\frac{u^2}{1+\rho^2}dV\leq\int_M|\nabla{u}|^2dV
\end{equation}
holds for some $C>0$ and for any $u\in\mathrm{Lip}_0(M)$.

\item[$(3)$]
If $\mu>2$ and \eqref{eq:eigenvalue_quadratic_decay} holds for all $r>0$, then the sharp Hardy type inequality
\begin{equation}\label{eq:Hardy_sharp}
\left(\frac{\mu-2}{2}\right)^2\int_M\frac{u^2}{\rho^2}dV \leq \int_M|\nabla{u}|^2dV
\end{equation}
holds for any $u\in\mathrm{Lip}_0(M)$.
\end{itemize}
\end{theorem}

\begin{remark}
$(a)$ It is well-known that if $M=\mathbb{R}^n$ and $\rho(x)=|x|$ then
\[
\lambda_1(B(0,r))=\frac{\lambda_n}{r^2}=\frac{j_{n/2-1}^2}{r^2},\ \ \ \forall\,r>0.
\]
Thus the volume growth in \eqref{eq:sharp_volume_growth} is sharp for $\mu=n$ and \eqref{eq:Hardy_sharp} reduces to the classical Hardy inequality for $\mu=n\geq3$.

$(b)$ \eqref{eq:sharp_volume_growth} also follows from the sharp Hardy type inequality (cf. Carron \cite{Carron19}, Proposition 2.26). Indeed, \eqref{eq:Hardy_sharp} implies the following reverse doubling property of order $\mu$ (cf. Lansade \cite{Lansade}, Proposition 5.2):
\[
\frac{|B_R|}{|B_r|}\geq{c_\mu}\frac{R^\mu}{r^\mu},\ \ \ \forall\,R>r>0.
\]
\end{remark}

Define
\[
\Lambda_*:=\liminf_{r\rightarrow +\infty} \{ r^2 \lambda_1(B_r)\}.  
\]
Note that $\mu\mapsto\lambda_\mu$ is a strict increasing continuous function when $\mu\geq2$ (cf. \cite{Elbert}). Thus a direct consequence of Theorem \ref{th:main} is the following

\begin{corollary}\label{cor:hyperbolic}
If\/ $\Lambda_*>\lambda_2=j_0^2\approx5.784$, then $M$ is hyperbolic. In other words, if $M$ is parabolic, then $\Lambda_*\leq{j_0^2}$.
\end{corollary}

\begin{remark}
We have already said that 
for $M=\mathbb{R}^2$ and $\rho(x)=|x|$ then
\[
\lambda_1(B(0,r))=\frac{\lambda_2}{r^2}=\frac{j_{0}^2}{r^2},
\]
so that $\Lambda_*=j^2_{0}$.   Since $\mathbb R^2$ is parabolic,  so the above result turns out to be  the best possible.
\end{remark}

Conversely, it is natural to ask whether there exists a universal constant $c_0$ such that $M$ is parabolic whenever $\Lambda_*\leq{c_0}$. The answer is, however, negative (see Example \ref{eg:counter_example} in \S\,\ref{sec:eg}).

Theorem \ref{th:main}/(1) allows us to estimate $\Lambda_*$ through volume growth conditions.  Cheng-Yau \cite{ChengYau75} showed that $\lambda_1(M)=0$ if $M$ has polynomial volume growth.   This was extended by Brooks \cite{Brooks},  who showed that 
if the volume $|M|$ of $M$ is infinite,  then
\[
\lambda_1(M) \le \frac{{\mu^*}^2}{4},\ \ \  \mu^*:=\limsup_{r\rightarrow +\infty} \frac{\log |B(x_0,r)|}{r}.
\]
Refined results for ends of  complete Riemannian manifolds are obtained by Li-Wang \cite{LiWang} (see also Carron \cite{Carron19}, \S\,2.4).  The following consequence of Theorem \ref{th:main} may be viewed as a quantitative version of the theorem of Cheng-Yau.

\begin{corollary}\label{cor:Upper}
If\/ $\nu_*:=\liminf_{r\rightarrow +\infty} \frac{\log |B_r|}{\log r}$,  then ${\Lambda_*} \leq\lambda_{\nu_*}=j_{\nu_*/2-1}^2$. In particular,  we have 
\begin{enumerate}
\item[$(1)$] $\Lambda_*=0$\/ if\/ $\nu_*=0;$
\item[$(2)$]
${\Lambda_*} \lesssim\nu_*$\/ if\/   $0<\nu_*\ll1;$
\item[$(3)$]
${\Lambda_*} \lesssim\nu_*^2$\/ if\/ $\nu_*\gg1$.
\end{enumerate}
\end{corollary}

When $|M|<+\infty$, we have $\nu_*=0$, so that $\lambda_1(B_r)$ decays faster than quadratically as $r\rightarrow+\infty$. More precisely, we have

\begin{proposition}\label{prop:volume_exponential}
If\/ $|M|<\infty$,  then  
\begin{equation}\label{eq:volume_exponential}
\widetilde{\Lambda}_*:=\liminf_{r\rightarrow+\infty}\frac{-\log\lambda_1(B_r)}{r}\geq \alpha_*:=\liminf_{r\rightarrow+\infty}\frac{-\log|M\setminus B_r|}{r}.
\end{equation}
\end{proposition}

It follows from Proposition \ref{prop:volume_exponential} that $\lambda_1(B_r)$ decays exponentially if $|M\setminus{B_r}|$ decays exponentially as $r\rightarrow+\infty$. On the other hand, the relationship of $\lambda_1(M\setminus{B_r})$ and $|M\setminus{B_r}|$ is studied by Brooks \cite{Brooksfinite}, who proved that if $|M|<\infty$,  then
$\lambda_1(M\setminus{B_r}) \le \frac{{\alpha^*}^2}{4}$ for all $r>0$, where
\[
\alpha^*:=\limsup_{r\rightarrow +\infty} \frac{-\log |M\setminus B_r|}{r}.
\]
A more precise version of Brooks' result is also proved by Li-Wang \cite{LiWang}.

Motivated by a result of Dodziuk-Pignataro-Randol-Sullivan \cite{DPRS}, we present an example in \S\,\ref{sec:eg} showing that the estimate in Proposition \ref{prop:volume_exponential} is sharp. Some other examples such that $\lambda_1(B_r)$ have various decaying behaviours are also given in \S\,\ref{sec:eg}.  In particular, Example \ref{eg:counter_example} shows that $\widetilde{\Lambda}_*>0$ does not necessarily imply $|M|<\infty$, i.e., the assumption that $M$ has finite volume in Proposition \ref{prop:volume_exponential} cannot be removed.

We also show that \eqref{eq:eigenvalue_quadratic_decay} holds under suitable condition on $\rho$. 

\begin{proposition}\label{prop:eigenvalue}
Suppose that $\rho$ is a nonnegative locally Lipschitz continuous exhaustion function on $M$ such that $|\nabla \rho|\le 1$ a.e. and $\Delta \rho^2 \geq 2\mu$ in the sense of distributions. Then
\begin{equation}\label{eq:eigenvalue_quadratic_decay_1}
\lambda_1(B_r)\geq\frac{\lambda_\mu}{r^2},\ \ \ \forall\,r>0.
\end{equation}
\end{proposition}

Proposition \ref{prop:eigenvalue} together with Theorem \ref{th:main} immediately yield the following

\begin{corollary}[cf. \cite{Carron97}]
If $\Delta \rho^2 \geq 2\mu>4$, then $M$ is hyperbolic and the sharp Hardy type inequality \eqref{eq:Hardy_sharp} holds.
\end{corollary}

Proposition \ref{prop:eigenvalue} implies several well-known results. First, if $M$ is an $n$-dimensional Cartan-Hadamard manifold and $\rho$ is the geodesic distance function, then Proposition \ref{prop:eigenvalue} together with the Hessian comparison theorem yield $\Lambda_*\geq\lambda_n$. In particular, it follows from Theorem \ref{th:main} that $M$ is hyperbolic when $n\geq3$ (cf. Ichihara \cite{Ichihara1,Ichihara2}, see also Grigor'yan \cite{Grigoryan}, Theorem 15.3). Next, if $M$ is a complete $n$-dimensional minimal submanifold in $\mathbb{R}^N$ and $\rho$ is the restriction of Euclidean distance to a given point $x_0\in\mathbb{R}^N$, then $\Delta\rho^2\geq2n$, so that
\begin{equation}\label{eq:minimal_submanifolds_lambda_1}
\lambda_1(B^N(x_0,r)\cap{M})\geq\frac{\lambda_n}{r^2},
\end{equation}
in view of Proposition \ref{prop:eigenvalue}. Here $B^N(x_0,r)$ is a Euclidean ball in $\mathbb{R}^N$. \eqref{eq:minimal_submanifolds_lambda_1} was first proved by Cheng-Li-Yau \cite{CLY} by heat kernel method. In particular, $M$ is hyperbolic for $n\geq3$, which can also be proved by the isoperimetric inequality of Michael-Simon \cite{MS} and Theorem 8.2 in \cite{Grigoryan}.

\subsection*{Comments}
In an early version of this paper,  the last two authors obtained the conclusion of Corollary \ref{cor:hyperbolic} under a worse condition $\Lambda_\ast > 18.624\cdots$. Shortly afterwards, the first author suggested some ideas for improving certain results. We then decided to write a joint paper on the subject, and the improvements found in this paper are the result of this collaboration.

\section{Proof of Theorem \ref{th:main}}\label{sec:main}
\subsection{Bessel's functions}

Assume that $\nu>-1$ and $\mu>0$. The Bessel function $J_\nu$ is given by
\[
J_\nu(t)=\sum^\infty_{m=0}\frac{(-1)^m}{m!\Gamma(m+\nu+1)}\left(\frac{t}{2}\right)^{2m+\nu},
\]
which is a solution of the ODE
\[
t^2J_\nu''(t)+tJ_\nu'(t)+(t^2-\nu^2)J_\nu(t)=0.
\]
Thus $\psi_\mu(s)=s^{1-\mu/2}J_{\mu/2-1}(\sqrt{\lambda_\mu}s)$ satisfies
\begin{equation}\label{eq:ODE_psi_mu}
\psi_\mu''(s)+\frac{\mu-1}{s}\psi_\mu'(s)+\lambda_\mu\psi_\mu(s)=0,
\end{equation}
where $\sqrt{\lambda_\mu}=j_{\mu/2-1}$ is the first positive zero of $J_{\mu/2-1}$. In particular, $\psi_\mu(1)=0$. Moreover, we have
\[
\psi_\mu(s)=\sum^\infty_{m=0}\frac{(-1)^m}{m!\Gamma(m+\mu/2)}\left(\frac{\sqrt{\lambda_\mu}s}{2}\right)^{2m},
\]
so that $\psi_\mu'(0)=0$.

Let us first verify the following

\begin{lemma}\label{lm:psi_mu_integral}
The following properties hold:
\begin{itemize}
\item[$(1)$]
For any $[a,b]\subset[0,1]$, we have
\begin{equation}\label{eq:intDeltacarre}
\int_a^b \left(\lambda_\mu\psi_\mu(s)^2-\psi_\mu'(s)^2\right)s^{\mu-1}ds=-\psi_\mu'(b)\psi_\mu(b)b^{\mu-1}+\psi_\mu'(a)\psi_\mu(a)a^{\mu-1}.
\end{equation}

\item[$(2)$] $\psi_\mu'(s)<0$ for all $s\in(0,1]$. More precisely,
\begin{equation}\label{eq:ODE_psi_mu_1}
\psi_\mu'(s)s^{\mu-1}=-\lambda_\mu\int^s_0\psi_\mu(t)t^{\mu-1}dt.
\end{equation}
\end{itemize}
\end{lemma}

\begin{proof}
By \eqref{eq:ODE_psi_mu}, we have
\begin{align*}
\left(\lambda_\mu\psi_\mu(s)^2-\psi_\mu'(s)^2\right)s^{\mu-1}
=&\,\left(-\psi_\mu''(s)\psi_\mu(s)-\frac{\mu-1}{s}\psi_\mu'(s)\psi_\mu(s)-\psi_\mu'(s)^2\right)s^{\mu-1}\\
=&\,-\left(\psi_\mu'(s)\psi_\mu(s)\right)'s^{\mu-1}-(\mu-1)\psi_\mu'(s)\psi_\mu(s)s^{\mu-2}\\
=&\,-\left(\psi_\mu'(s)\psi_\mu(s)s^{\mu-1}\right)'
\end{align*}
and
\[
-\lambda_\mu\psi_\mu(s)s^{\mu-1}=\psi''_\mu(s)s^{\mu-1}+(\mu-1)\psi_\mu'(s)s^{\mu-2}=\left(\psi_\mu'(s)s^{\mu-1}\right)',
\]
from which \eqref{eq:intDeltacarre} and \eqref{eq:ODE_psi_mu_1} follow immediately.
\end{proof}
\subsection{The volume growth}
\begin{proof}[Proof of Theorem \ref{th:main}/(1)]
We first assume that $\rho>0$. Set
$
\phi=\psi_\mu\left(\rho/r\right),
$
where $\psi_\mu$ is given as \eqref{eq:psi_mu}.  Then the variational characterization of eigenvalue gives
\[
\lambda_\mu\int_{B_r}\psi_\mu(\rho/r)^2dV \leq r^2\lambda_1(B_r)\int_{B_r}\phi^2dV  \leq r^2\int_{B_r}\left|\nabla\phi\right|^2dV \leq \int_{B_r}\psi'_\mu(\rho/r)^2dV.
\]
By using the co-area formula, this can be rewritten as
\[
\lambda_\mu\int_0^r\psi_\mu(t/r)^2 d\sigma(t) \leq \int_0^r \psi'_\mu(t/r)^2 d\sigma(t),
\]
where $d\sigma(t):=(\rho)_{\#}(dV)$ is a Lebesgue-Stieltjes measure on $(0,+\infty)$.  Divide this inequality by $r^{\mu+1}$ and integrate on $r\in [r_0,\overline{r}]$,  we obtain 
\[
\lambda_\mu\int_0^{\overline{r}}\int_{\max\{r_0,t\}}^{\overline{r}} \psi_\mu(t/r)^2\frac{dr}{r^{\mu+1}} \, d\sigma(t)\leq \int_0^{\overline{r}}\int_{\max\{r_0,t\}}^{\overline{r}} \psi'_\mu(t/r)^2\frac{dr}{r^{\mu+1}} d\sigma(t),
\]
in view of Fubini's theorem. 
By using the new variable $s=t/r$, we get 
\begin{equation}\label{eq:volume_growth_1}
\lambda_\mu\int_0^{\overline{r}}\left(\int_{t/\overline{r}}^{\min\{1,t/r_0\}}\psi_\mu(s)^2s^{\mu-1}ds\right)\frac{d\sigma(t)}{t^{\mu}} \leq \int_0^{\overline{r}}\left(\int_{t/\overline{r}}^{\min\{1,t/r_0\}}\psi'_\mu(s)^2s^{\mu-1}ds\right)\frac{d\sigma(t)}{t^{\mu}}.
\end{equation}
Take $a=t/\overline{r}$ and $b=\min\{1,t/r_0\}$ when $t\leq{r_0}$ in \eqref{eq:intDeltacarre}, we infer from \eqref{eq:volume_growth_1} and the facts $\psi_\mu(1)=0$ and $\psi_\mu'\leq0$ that
\begin{align*}
&\,\int_0^{\overline{r}} -\psi'_\mu(t/\overline{r})\psi_\mu(t/\overline{r})\left(t/\overline{r}\right)^{\mu-1}\frac{d\sigma(t)}{t^{\mu}}\\
\geq&\,\int_0^{\overline{r}} -\psi'_\mu(\min\{1,t/r_0\})\psi_\mu(\min\{1,t/r_0\})(\min\{1,t/r_0\})^{\mu-1}\frac{d\sigma(t)}{t^{\mu}}\\
\geq&\,\int_0^{r_0} -\psi'_\mu(t/r_0)\psi_\mu(t/r_0)(t/r_0)^{\mu-1}\frac{d\sigma(t)}{t^{\mu}},
\end{align*}
i.e.,
\begin{equation}\label{eq:volume_growth_2}
\frac{1}{\overline{r}^\mu}\int_{B_{\overline{r}}} -\frac{\psi'_\mu(\rho/\overline{r})\psi_\mu(\rho/\overline{r})}{\rho/\overline{r}}dV \geq \frac{1}{r_0^\mu}\int_0^{r_0} -\frac{\psi'_\mu(\rho/r_0)\psi_\mu(\rho/r_0)}{\rho/r_0}dV.
\end{equation}

If we merely have $\rho\geq0$, then we may also apply the above argument to
\[
\rho_\varepsilon:=\sqrt{\rho^2+\varepsilon^2}.
\]
Note that we still have
\[
|\nabla\rho_\varepsilon|=\frac{\rho|\nabla\rho|}{(\rho^2+\varepsilon^2)^{1/2}}\leq1.
\]
Since $B_r^\varepsilon:=\{\rho_\varepsilon<r\}=B_{\sqrt{r^2-\varepsilon^2}}$, it follows that
\begin{equation}\label{eq:eigen_epsilon}
\lambda_1(B^\varepsilon_r)=\lambda_1\left(B_{\sqrt{r^2-\varepsilon^2}}\right)\geq\frac{\lambda_\mu}{r^2-\varepsilon^2}\geq\frac{\lambda_\mu}{r^2},\ \ \ \forall\,r\geq{r_{0,\varepsilon}}:=\sqrt{r_0^2+\varepsilon^2},
\end{equation}
so that \eqref{eq:volume_growth_2} still holds with $\rho$ and $r_0$ replaced by $\rho_\varepsilon$ and $r_{0,\varepsilon}$, respectively. Moreover, since $\psi_\mu$ is a $C^2$ function on $[0,1]$ with $\psi_\mu'(0)=0$, there exists a constant $A>0$ such that
\[
-\psi_\mu'(s)\psi_\mu(s)\le As,\ \ \ \forall\,s\in[0,1].
\]
Letting $\varepsilon\rightarrow0+$, we see that \eqref{eq:volume_growth_2} remains valid for $\rho$ and $r_0$ in view of the dominated convergence theorem, which yields
\[
\frac{|B_{\overline{r}}|}{\overline{r}^\mu} \geq \frac{1}{Ar_0^\mu}\int_0^{r_0} -\frac{\psi'_\mu(t/r_0)\psi_\mu(t/r_0)}{t/r_0}d\sigma(t) =:c,\ \ \ \forall\,\overline{r}\geq{r_0}. \qedhere
\]
\end{proof}

\begin{proof}[Proof of Corollary \ref{cor:Upper}]
Suppose on the contrary that $\Lambda_*>\lambda_{\nu_*}$. Since the function $\mu\mapsto\lambda_\mu$ is continuous, we have $\Lambda_*>\lambda_{\nu_*+\varepsilon}$ for  $\varepsilon\ll 1$,  so that $\lambda_1(B_r)\geq\lambda_{\nu_*+\varepsilon}/r^2$ for $r\gg1$.  By Theorem \ref{th:main}/(1), we conclude that $|B_r|\gtrsim r^{\nu_*+\varepsilon}$.  But this implies $\nu_*\geq\nu_*+\varepsilon$,  which is  impossible.

Since 
$
j_\nu\sim2\sqrt{\nu+1}
$
as $\nu\rightarrow-1+$  (cf. Piessens \cite{Piessens}),  we have
 $\lambda_{\nu_*}=j_{\nu_*/2-1}^2\sim2\nu_*$ as $\nu_*\rightarrow0+$,  from which assertions (1) and (2) immediately follow.  On the other hand,  since $j_\nu\sim\nu$ as $\nu\rightarrow+\infty$ (cf. Watson \cite{Watson}, pp. 521, see also Elbert \cite{Elbert}, \S\,1.4),  we have 
$\lambda_{\nu_*}=j_{\nu_*/2-1}^2\sim\nu_*^2/4$ as $\nu_*\rightarrow+\infty$, which implies (3).
\end{proof}

\subsection{The Hardy type inequalities}\label{subsec:Hardy}

Recall that the capacity $\mathrm{cap}(K)$ of a compact set $K\subset M$ is given by
$$
\mathrm{cap}(K):=\inf \int_M |\nabla \psi|^2 dV,
$$
where the infimum is taken over all $\psi\in\mathrm{Lip}_0(M)$ with $0\le \psi\le 1$ and $\psi|_K=1$.  The following criterion is of fundamental importance (cf. Grigor'yan \cite{Grigoryan}, Theorem 5.1 and Ancona \cite{Ancona}, pp. 46--47, see also Carron \cite{Carron07}, Definition 2.13).

\begin{theorem}[]\label{th:Capacity}
Let $(M,g)$ be a complete Riemannian manifold. Then the following properties are equivalent:

\begin{itemize}
\item[$(1)$]
$M$ is hyperbolic.

\item[$(2)$]
$\mathrm{cap}(K)>0$ for some/any compact set $K\subset M$ with non-empty interior.  

\item[$(3)$]
Given some/any open subset $U\subset\subset{M}$, there exists a constant $C=C(U)$ such that
\[
\int_Uu^2dV\leq{C(U)}\int_M|\nabla{u}|^2dV
\]
for any $u\in\mathrm{Lip}(M)$ with a compact support.
\end{itemize}

\end{theorem}

We shall take a unified approach to proving the Hardy inequalities \eqref{eq:HardyIneq} and \eqref{eq:Hardy_sharp}. To begin with, set
\[
\Phi(s)=J_{\mu/2-1}\left(\sqrt{\lambda_\mu}s\right).
\]
The function $\Phi$ satisfies $\Phi(0)=\Phi(1)=0$ and
\begin{equation}\label{eq:Bessel_equation_Phi}
\Phi''(s)+\frac1s \Phi'(s)+\left(\lambda_\mu-\frac{(\mu/2-1)^2}{s^2}\right)\Phi(s)=0.
\end{equation}

We shall make use the following property of $\Phi$.

\begin{lemma}\label{lm:intPhi}
For any $s\in[0,1]$, we have
\begin{equation}\label{eq:intPhi}
\int_0^x \Phi'(s)^2 sds=\lambda_\mu \int_0^x \Phi(s)^2sds-\left(\frac{\mu-2}{2}\right)^2 \int_0^x \Phi(s)^2\frac{ds}{s}+\Phi'(x)\Phi(x)x.
\end{equation}
In particular,
\begin{equation}\label{eq:intPhi_1}
\int_0^1 \Phi'(s)^2 sds=\lambda_\mu \int_0^1 \Phi(s)^2sds-\left(\frac{\mu-2}{2}\right)^2 \int_0^1 \Phi(s)^2\frac{ds}{s}.
\end{equation}
\end{lemma}

\begin{proof}
By \eqref{eq:Bessel_equation_Phi}, we have
\begin{align*}
\lambda_\mu\Phi(s)^2s-\nu^2\Phi(s)^2\frac{1}{s}
=&\,\Phi''(s)\Phi(s)s+\Phi'(s)\Phi(s)\\
=&\,\left(\Phi'(s)\Phi(s)s\right)'-\Phi'(s)^2s,
\end{align*}
from which \eqref{eq:intPhi} and \eqref{eq:intPhi_1} follow immediately.
\end{proof}

\begin{proof}[Proof of Theorem \ref{th:main}/(2)(3)]

As the proof of (1), we first consider the case that $\rho>0$. Given $u\in\mathrm{Lip}_0(M)$, define
\[
\phi(x):=u(x)\Phi\left(\rho(x)/r\right),\ \ \ x\in{M}.
\]
Since $|\nabla\rho|\leq1$, we have
\[
\left|\nabla\phi\right|^2\le \Phi\left(\rho/r\right)^2\left|\nabla{u}\right|^2 + \frac{1}{r^2}u^2\Phi'\left(\rho/r\right)^2+2\frac{u}{r}\langle \nabla u,\nabla \rho\rangle \Phi'\left(\rho/r\right)\Phi\left(\rho/r\right).
\]
The variational characterization of eigenvalue gives
\begin{align}\label{eq:Hardy_main_1}
\frac{\lambda_\mu}{r^2} \int_{B_r} \phi^2dV
\leq&\,\lambda_1(B_r)\int_{B_r} \phi^2 dV \leq \int_{B_r} |\nabla\phi|^2dV\notag\\
\leq&\,\int_{B_r}|\nabla{u}|^2\Phi(\rho/r)^2 dV+\frac{1}{r^2} \int_{B_r}  u^2\Phi'(\rho/r)^2dV\notag\\
&\,+ 2\int_{B_r} \frac{u}{r}\langle \nabla u,\nabla \rho\rangle \Phi'\left(\rho/r\right)\Phi\left(\rho/r\right)dV,\ \ \ \forall\,r>r_0.
\end{align}
We then divide the \eqref{eq:Hardy_main_1} by $r$ and integrate for $r\in (r_0,+\infty)$. (Under the condition of (3) we set $r_0=\inf_M\rho>0$.) First,
\begin{align*}
&\,\int_{r_0}^{+\infty} \int_{B_r} u(x)^2\Phi\left(\rho(x)/r\right)^2 dV(x) \frac{dr}{r^3}\\
=&\, \int_M \left(\int^{+\infty}_{\max\{r_0, \rho(x)\}} \Phi\left(\rho(x)/r\right)^2\frac{dr}{r^3}\right)u(x)^2dV(x)\\
=&\,\int_M \left( \int_{0}^{\min\{1,  \rho(x)/r_0\}} \Phi\left(s\right)^2sds\right)\frac{u(x)^2}{\rho(x)^2}dV(x)\\
\geq&\,\left( \int_{0}^{1} \Phi\left(s\right)^2sds\right) \int_{M\setminus B_{r_0}}\frac{u^2}{\rho^2}dV,
\end{align*}
where we used the new variable $s=\rho(x)/r$ in the second step. Analogously, we have
\begin{align*}
&\,\int_{r_0}^{+\infty} \int_{B_r} |\nabla{u}(x)|^2\Phi(\rho(x)/r)^2 dV(x)\frac{dr}{r}\\
=&\,\int_M \left(\int^{+\infty}_{\max\{r_0, \rho(x)\}} \Phi^2\left(\rho(x)/r\right) \frac{dr}{r}\right) |\nabla{u}(x)|^2 dV(x)\\
=&\,\int_M \left(\int_{0}^{\min\{1,\rho(x)/r_0\}} \Phi^2(s)\frac{ds}{s}\right) |\nabla{u}(x)|^2 dV(x)\\
\leq&\,\left(\int_{0}^{1} \Phi^2(s)\frac{ds}{s}\right) \int_M|\nabla{u}|^2 dV
\end{align*}
and
\begin{align*}
&\,\, \int_{r_0}^{+\infty} \int_{B_r} u(x)^2 \Phi'(\rho(x)/r) ^2dV(x)\frac{dr}{r^3}\\
=\,&\,\int_M \left(\int^{+\infty}_{\max\{r_0,\rho(x)\}}\Phi'(\rho(x)/r)^2 \frac{dr}{r^3}\right) u(x)^2 dV(x)\\
=\,&\,\int_M \left(\int_{0}^{\min\{1, \rho(x)/r_0\}} \Phi'(s)^2s ds\right) \frac{u^2(x)}{\rho(x)^2} dV(x)\\
\leq\,&\,\left(\int^1_0 \Phi'(s)^2s ds\right) \int_{M\setminus B_{r_0}} \frac{u^2}{\rho^2} dV+\int_{ B_{r_0}}\left(\int_{0}^{\rho(x)/r_0} \Phi'(s)^2s ds\right) \frac{u^2}{\rho^2} dV\\
=:&\,\left(\int^1_0 \Phi'(s)^2s ds\right) \int_{M\setminus B_{r_0}} \frac{u^2}{\rho^2} dV+I_1(r_0),
\end{align*}
where
\begin{equation}\label{eq:I1}
I_1(r_0)\leq\frac{\left\|\Phi'\right\|_{L^\infty([0,1])}^2}{2r^2_0}\int_{ B_{r_0}} u^2 dV,\ \ \ \forall\,r_0\geq\inf_M\rho,
\end{equation}
so that $I_1(r_0)=0$ if $r_0=\inf_M\rho$. Eventually 
\begin{align*}
&\,\, \int_{r_0}^{+\infty} \int_{B_r} \frac{u}{r}\langle \nabla u,\nabla \rho\rangle \Phi'\left(\rho/r\right)\Phi\left(\rho/r\right)dV\frac{dr}{r}\\
=\,&\,\int_M \left(\int^{+\infty}_{\max\{r_0,\rho(x)\}}\Phi'(\rho(x)/r)\Phi\left(\rho(x)/r\right) \frac{dr}{r^2}\right) u(x)\langle \nabla u(x),\nabla \rho(x)\rangle dV(x)\\
=\,&\,\int_M \left(\int_{0}^{\min\{1,\rho(x)/r_0\}}\Phi'(s)\Phi\left(s\right) ds\right) \frac{u(x)}{\rho(x)}\langle \nabla u(x),\nabla \rho(x)\rangle dV(x)\\
=\,&\,\int_{B_{r_0}} \Phi\left(\rho/r_0\right)^2\frac{u}{\rho}\langle \nabla u,\nabla \rho\rangle dV=:I_2(r_0),
\end{align*}
where we have used the fact $2\int_{0}^{1}\Phi'(s)\Phi\left(s\right) ds=\Phi^2(1)-\Phi^2(0)=0$. We also have $I_2(r_0)=0$ if $r_0=\inf_M\rho$. Moreover, since $\Phi(0)=0$ we get that $\Phi^2(s)\le \left\|\Phi'\right\|^2_{L^\infty([0,1])} s^2$ for all $s\ge 0$, so that,
\begin{equation}\label{eq:I2}
I_2(r_0)
\leq\,\frac{\|\Phi'\|^2_{L^\infty([0,1])}}{r_0}\int_{B_{r_0}}| \nabla u||u| dV,\ \ \ \forall\,r_0\geq\inf_M\rho
\end{equation}
These together with \eqref{eq:Hardy_main_1} yield
\begin{align*}
&\,\lambda_\mu\left( \int_{0}^{1} \Phi\left(s\right)^2sds\right) \int_{M\setminus B_{r_0}}\frac{u^2}{\rho^2}dV\\
\leq&\,\left(\int_{0}^{1} \Phi^2(s)\frac{ds}{s}\right) \int_M|\nabla{u}|^2 dV+\left(\int^1_0 \Phi'(s)^2s ds\right) \int_{M\setminus B_{r_0}} \frac{u^2}{\rho^2} dV\\
&\,+I_1(r_0)+I_2(r_0).
\end{align*}
Thus
\begin{equation}\label{eq:Hardy_main_2}
\left(\frac{\mu-2}{2}\right)^2\int_{M\setminus B_{r_0}}\frac{u^2}{\rho^2}dV\leq \int_M|\nabla{u}|^2 dV+\frac{I_1(r_0)+I_2(r_0)}{A}
\end{equation}
in view of \eqref{eq:intPhi_1}, where $A=\int_{0}^{1} \Phi^2(s)\frac{ds}{s}$. If \eqref{eq:eigenvalue_quadratic_decay} holds for all $r>0$, we may take $r_0=\inf_M\rho$, so that \eqref{eq:Hardy_sharp} follows immediately from \eqref{eq:Hardy_main_2}.

Under the condition of (2), we infer from \eqref{eq:I1}, \eqref{eq:I2} and \eqref{eq:Hardy_main_2} that
\begin{equation}\label{eq:Hardy_main_3}
\left(\frac{\mu-2}{2}\right)^2\int_{M}\frac{u^2}{1+\rho^2}dV\leq (1+\delta)\int_M|\nabla{u}|^2 dV+C_\delta\int_{ B_{r_0}} u^2 dV
\end{equation}
holds for any $u\in\mathrm{Lip}_{\mathrm{loc}}(M)$ with a compact support and $\delta>0$, where 
\begin{equation}\label{eq:C_delta}
C_\delta=\left(\frac{\mu-2}{2}\right)^2+\frac{\|\Phi'\|_{L^\infty([0,1])}^2}{Ar_0^2}+\frac{\|\Phi'\|_{L^\infty([0,1])}^4}{A^2r_0^2\delta}.
\end{equation}
\
Note that $C_\delta$ only depends on $\mu$ and $r_0$ if $\delta$ is fixed.

The above proofs of inequalities \eqref{eq:Hardy_sharp} and \eqref{eq:Hardy_main_3} require an additional condition $\rho>0$. In general, if $\inf_M\rho=0$, we consider $\rho_\varepsilon:=\sqrt{\rho^2+\varepsilon^2}$. Analogously to the proof of Theorem \ref{th:main}/(1), we have \eqref{eq:eigen_epsilon}, so that \eqref{eq:Hardy_main_3} becomes
\[
\left(\frac{\mu-2}{2}\right)^2\int_{M}\frac{u^2}{1+\varepsilon^2+\rho^2}dV\leq (1+\delta)\int_M|\nabla{u}|^2 dV+C_{\delta,\varepsilon}\int_{ B_{r_{0}^\varepsilon}} u^2 dV,
\]
where $C_{\delta,\varepsilon}$ is given by \eqref{eq:C_delta} with $r_0^2$ replaced by $r_0^2+\varepsilon^2$. Thus \eqref{eq:Hardy_main_3} follows by letting $\varepsilon\rightarrow0$. Moreover, if \eqref{eq:eigenvalue_quadratic_decay} holds for all $r>0$, then $\lambda_1(B_r^\varepsilon)\geq\lambda_\mu/r^2$ when $r>\varepsilon$. Take $r_0=\varepsilon=\inf_M\rho_\varepsilon$ in \eqref{eq:Hardy_main_2} with $\rho$ replaced by $\rho_\varepsilon$, we have
\begin{equation}\label{eq:Hardy_sharp_epsilon}
\left(\frac{\mu-2}{2}\right)^2\int_M\frac{u^2}{\rho_\varepsilon^2}{dV}\leq\int_M|\nabla{u}|^2dV.
\end{equation}
We obtain \eqref{eq:Hardy_sharp} immediately by letting $\varepsilon\rightarrow0$, which completes the proof of (3).

To prove (2), we shall first derive the hyperbolicity of $M$ from \eqref{eq:Hardy_main_3} when $\mu>2$. By Theorem \ref{th:main}/(1), there exist some constants $c>0$ and $\mu>2$ such that $|B_r|\geq{cr^\mu}$ for $r\gg1$. Thus
\[
\int_{M} \frac{dV}{1+\rho^2}\geq\limsup_{r\rightarrow+\infty}\frac{|B_r|}{1+r^2}=+\infty.
\]
Choose $\overline{r}\gg{r_0}$ so that
\[
\left(\frac{\mu-2}{2}\right)^2\int_{B_{\overline{r}}}\frac{dV}{1+\rho^2}\geq C_\delta|B_{r_0}|+1+\delta.
\]
Then
\[
\int_M|\nabla{u}|^2dV\geq1
\]
whenever $u\in\mathrm{Lip}_0(M)$ and $u=1$ on $\overline{B}_{\overline{r}}$. Thus $\mathrm{cap}(\overline{B}_{\overline{r}})\geq1$ and $M$ is hyperbolic in view of Theorem \ref{th:Capacity}/(2).

Finally, it follows from Theorem \ref{th:Capacity}/(3) that
\[
\int_{B_{r_0}}u^2dV\leq{C(B_{r_0})}\int_M|\nabla{u}|^2dV
\]
for any $u\in\mathrm{Lip}_0(M)$. This together with \eqref{eq:Hardy_main_3} give \eqref{eq:HardyIneq} with
\[
C=\left(\frac{\mu-2}{2}\right)^2\frac{1}{1+\delta+C_\delta C(B_{r_0})}.
\]
If we fix some $\delta>0$, then $C$ only depends on $\mu$, $r_0$ and the geometry of $M$.
\end{proof}

\section{Proof of Proposition \ref{prop:volume_exponential}}
By definition, there exists a sequence $\{r_k\}$ with $\lim_{k\rightarrow+\infty}r_k=+\infty$, such that $\lambda_1(B_{r_k})>e^{-(\widetilde{\Lambda}_*+\varepsilon)r_k}$ for some $0<\varepsilon\ll1$. Again, for $k\geq1$ and $0<\delta<1$, we take a cut-off function $\phi_k:M\rightarrow[0,1]$ such that $\phi_k|_{B_{\delta{r_k}}}=1$,  $\phi_k|_{M\setminus{B_{r_k}}}=0$ and $|\nabla\phi_k|\leq(1-\delta)^{-1}r_k^{-1}$. Then
\begin{align*}
e^{-(\widetilde{\Lambda}_*+\varepsilon)r_k}\,|B_{\delta{r_k}}|
\leq&\,\lambda_1(B_{r_k}) \int_M \phi_k^2 dV \leq \int_M |\nabla\phi_k|^2 dV\\
\leq&\,\frac{1}{(1-\delta)^2r_k^2}|B_{r_k}\setminus B_{\delta{r_k}}|\\
\leq&\,\frac{1}{(1-\delta)^2r_k^2}|M\setminus B_{\delta{r_k}}|.
\end{align*}
That is,
\[
|M|\geq\left(1+e^{-(\widetilde{\Lambda}_*+\varepsilon)r_k}(1-\delta)^2r_k^2\right)|B_{\delta{r_k}}|,
\]
which is equivalent to
\[
|M\setminus{B_{\delta{r_k}}}|\geq\frac{e^{-(\widetilde{\Lambda}_*+\varepsilon)r_k}(1-\delta)^2r_k^2}{1+e^{-(\widetilde{\Lambda}_*+\varepsilon)r_k}(1-\delta)^2r_k^2}|M|.
\]
Thus
\[
\alpha_*\leq\lim_{k\rightarrow\infty}\frac{-\log|M\setminus B_{\delta{r_k}}|}{\delta{r_k}}\leq\frac{\widetilde{\Lambda}_*+\varepsilon}{\delta}.
\]
Letting $\delta\rightarrow1-$ and $\varepsilon\rightarrow0+$, we conclude that $\widetilde{\Lambda}_*\geq\alpha_*$.

\section{Proof of Proposition \ref{prop:eigenvalue}}\label{sec:Laplacian_rho}

Let $w\in\mathrm{Lip}_{\mathrm{loc}}(M)$ and $v\in{L^1_{\mathrm{loc}}(M)}$, by $\Delta{w}\geq{v}$ in the sense of distributions (or weakly) for some locally integrable function $v$, we mean
\begin{equation}\label{eq:distribution}
\int_M\langle \nabla w,\nabla \varphi\rangle{dV}\leq-\int_Mv\varphi{dV}
\end{equation}
for any nonnegative function $\varphi\in{C^\infty_0(M)}$. Since $C^\infty_0(M)$ is dense in $\mathrm{Lip}_0(M)$ with respect to the Sobolev norm $\|\cdot\|_{W^{1,2}}$, we see that $\Delta{w}\geq{v}$ weakly if and only if \eqref{eq:distribution} holds for all nonnegative $\varphi\in\mathrm{Lip}_0(M)$. One can define $\Delta{w}\leq{v}$ in a similar way.

We first prove the following technical lemma:

\begin{lemma}\label{lm:Laplacian_weak}
Suppose that $\Delta{w}\geq{v}$ weakly and $w>0$. Let $f:(0,+\infty)\rightarrow\mathbb(0,+\infty)$ be a smooth function. Then the following properties hold:
\begin{itemize}
\item[$(1)$]
If $f$ is increasing, then $\Delta{f(w)} \geq f'(w)v+f''(w)|\nabla{w}|^2$ weakly;

\item[$(2)$]
If $f$ is decreasing, then $\Delta{f(w)} \leq f'(w)v+f''(w)|\nabla{w}|^2$ weakly.
\end{itemize}
\end{lemma}

\begin{proof}
(1) Let $\varphi\in\mathrm{Lip}_0(M)$ and $\varphi\geq0$. If $f$ is increasing, then we also have $f'(w)\varphi\in\mathrm{Lip}_0(M)$ and $f'(w)\varphi\geq0$. Thus
\begin{align*}
\int_M\left\langle{\nabla{f(w)},\nabla\varphi}\right\rangle{dV}
=&\,\int_M\left\langle{f'(w)\nabla{w},\nabla\varphi}\right\rangle{dV}\\
=&\,\int_M\left\langle{\nabla{w},\nabla\left(f'(w)\varphi\right)}\right\rangle{dV}-\int_Mf''(w)|\nabla{w}|^2\varphi{dV}\\
\leq&\,-\int_M\left(f'(w)v+f''(w)|\nabla{w}|^2\right)\varphi{dV},
\end{align*}
which proves the first assertion.

(2) Analogously, we have
\begin{align*}
\int_M\left\langle{\nabla{f(w)},\nabla\varphi}\right\rangle{dV}
=&\,\int_M\left\langle{f'(w)\nabla{w},\nabla\varphi}\right\rangle{dV}\\
=&\,-\int_M\left\langle{\nabla{w},\nabla\left(-f'(w)\varphi\right)}\right\rangle{dV}-\int_Mf''(w)|\nabla{w}|^2\varphi{dV}\\
\geq&\,-\int_M\left(f'(w)v+f''(w)|\nabla{w}|^2\right)\varphi{dV}.\qedhere
\end{align*}
\end{proof}

\begin{proof}[Proof of Proposition \ref{prop:eigenvalue}]
We follow the ideas in \cite{Carron97}. Assume for a moment that $\rho>0$. Apply Lemma \ref{lm:Laplacian_weak}/(1) with $w=\rho^2$ and $f(t)=t^{1/2}$, we get
\begin{equation}\label{eq:Delta_rho}
\Delta\rho\geq\frac{\mu-|\nabla\rho|^2}{\rho}
\end{equation}
weakly.

Let $\phi$ be a Lipschitz continuous function which is positive on $B_r$ and fix $u\in\mathrm{Lip}_0(B_r)$. With $v:=u/\phi$, we have
\begin{align}\label{eq:ineq_u_vphi}
\int_M|\nabla{u}|^2dV
=&\,\int_M\phi^2|\nabla{v}|^2{dV}+\int_M{v}^2|\nabla\phi|^2{dV}+2\int_M\phi{v}\langle{\nabla\phi,\nabla{v}}\rangle{dV}\notag\\
=&\,\int_M\phi^2|\nabla{v}|^2dV+\int_M\left\langle{\nabla(\phi{v}^2),\nabla\phi}\right\rangle{dV}.
\end{align}

Let us choose $\phi=\psi_\mu(\rho/r)$, where
\[
\psi_\mu(s)=s^{1-\mu/2}J_{\mu/2-1}(\sqrt{\lambda_\mu}s).
\]
Recall that $\psi_\mu$ is a solution of the ODE \eqref{eq:ODE_psi_mu} and $\psi_\mu'\leq0$. Thus \eqref{eq:Delta_rho},  \eqref{eq:ODE_psi_mu} together with Lemma \ref{lm:Laplacian_weak}/(2) yield
\begin{align}\label{eq:Delta_psi_mu}
\Delta\phi
\leq&\,\psi_\mu'(\rho/r)\frac{\mu-|\nabla\rho|^2}{\rho{r}}+\psi_\mu''(\rho/r)\frac{|\nabla\rho|^2}{r^2}\notag\\
=&\,\psi_\mu'(\rho/r)\frac{\mu-1}{\rho{r}}+\psi_\mu''(\rho/r)\frac{1}{r^2}+\psi_\mu'(\rho/r)\frac{1-|\nabla\rho|^2}{\rho{r}}-\psi_\mu''(\rho/r)\frac{1-|\nabla\rho|^2}{r^2}\notag\\
=&\,-\frac{\lambda_\mu}{r^2}\phi+\frac{1-|\nabla\rho|^2}{r^2}\left(\frac{\psi_\mu'(\rho/r)}{\rho/r}-\psi_\mu''(\rho/r)\right).
\end{align}
We have (cf. Lebedev \cite[formula (5.3.5)]{Lebedev})
\[
\psi_\mu'(s)=-s\psi_{\mu+2}(s),
\]
from which it follows that
\[
\psi_\mu'(s)/s-\psi_\mu''(s)=-\psi_{\mu+2}(s)-\left(-s\psi_{\mu+2}(s)\right)'=-s^2\psi_{\mu+4}(s)\leq0.
\]
Then \eqref{eq:Delta_psi_mu} implies
\[
\Delta\phi\leq-\frac{\lambda_\mu}{r^2}\phi,
\]
which gives
\[
\int_M\left\langle{\nabla(\phi{v}^2),\nabla\phi}\right\rangle{dV}\geq-\int_M\phi{v}^2\left(-\frac{\lambda_\mu}{r^2}\phi\right)dV=\frac{\lambda_\mu}{r^2}\int_M{u}^2{dV}.
\]
This together with \eqref{eq:ineq_u_vphi} yield
\begin{equation}\label{eq:Hardy_Poincare}
\frac{\lambda_\mu}{r^2}\int_{B_r}u^2dV \leq \int_{B_r}|\nabla{u}|^2dV,\ \ \ \forall\,u\in\mathrm{Lip}_{\mathrm{loc}}(B_r)
\end{equation}
under the additional condition $\rho>0$.

In general, we use $\rho_\varepsilon:=(\rho^2+\varepsilon^2)^{1/2}$ instead of $\rho$. Note that $|\nabla\rho_\varepsilon|\leq1$ and
\[
\Delta\rho_\varepsilon^2=\Delta\rho^2\geq2\mu.
\]
Thus \eqref{eq:Hardy_Poincare} holds if $B_r$ is replaced by $B_r^\varepsilon$ in view of \eqref{eq:eigen_epsilon}. The proposition follows by letting $\varepsilon\rightarrow0$.
\end{proof}

\section{Examples}\label{sec:eg}

Let $M=\mathbb{R}\times\mathbb{S}^1$ be equipped with the following Riemannian metric
\[
g=dt^2+\eta'(t)^2d\theta^2,\ \ \ t\in\mathbb{R},\ e^{i\theta}\in\mathbb{S}^1,
\]
where $\eta:\mathbb{R}\rightarrow\mathbb{R}$ is a smooth function such that $\eta'(t)>0$ and $\lim_{t\rightarrow-\infty}\eta(t)=0$. Dodziuk-Pigmataro-Randol-Sullivan \cite[Proposition 3.1]{DPRS} showed that if $\eta(t)=e^t$, then $\lambda_1(M)\geq1/4$, so that $(M,g)$ is hyperbolic.

In general, we consider
\[
h(t):=\int^t_0\frac{ds}{\eta'(s)},\ \ \ t\in\mathbb{R},
\]
which gives a function on $M$. A straightforward calculation shows
\[
\Delta{h}=\frac{\partial^2h}{\partial{t}^2}+\frac{\eta''(t)}{\eta'(t)}\frac{\partial{h}}{\partial{t}}+\frac{1}{\eta'(t)^2}\frac{\partial^2h}{\partial\theta^2}=0,
\]
i.e., $h$ is a harmonic function on $M$. Moreover, in the new coordinate system $(\widetilde{t},\theta):=(h(t),\theta)$, we may write
\[
g=\eta'(t)^2\left(h'(t)^2dt^2+d\theta^2\right) =\eta'(t)^2\left(d\widetilde{t}\,^2+d\theta^2\right),
\]
which indicates that $(M,g)$ is conformally equivalent to the cylinder $(\inf{h},\sup{h})\times\mathbb{S}^1$ equipped with a flat metric. In conclusion, we have

\begin{proposition}\label{prop:example_hyperbolic}
$(M,g)$ is hyperbolic if and only if $\inf{h}>-\infty$ or $\sup{h}<+\infty$.
\end{proposition}

Let $\rho(t,\theta)=|t|$. Clearly, $\rho$ is an exhaustion function which satisfies $|\nabla\rho|_g\leq1$. Indeed, $\rho$ is the geodesic distance to the circle $\{0\}\times\mathbb{S}^1$. The goal of this section is to investigate the asymptotic behaviour of $\lambda_1(B_r)$ as $r\rightarrow+\infty$ for different choices of $\eta$. We start with the following elementary lower estimate.

\begin{proposition}\label{prop:example}
\[
\lambda_1(B_r)\geq\frac{1}{4}\inf_{|t|\leq{r}}\frac{\eta'(t)^2}{\eta(t)^2}.
\]
\end{proposition}

\begin{proof}
The idea is essentially implicit in \cite{DPRS}. Since $dV=\eta'(t)dtd\theta$, we have
\[
\int^r_{-r}\phi^2\eta'(t)dt=-2\int^r_{-r}\phi\frac{\partial\phi}{\partial{t}}\eta(t)dt,\ \ \ \forall\,\phi\in{C^\infty_0(B_r)},
\]
so that
\begin{align}\label{eq:Cauchy_Schwarz_example}
\int^r_{-r}\phi^2\eta'(t)dt
\leq&\,4\int^r_{-r}\left|\frac{\partial\phi}{\partial{t}}\right|^2\frac{\eta(t)^2}{\eta'(t)}dt\leq4\int^r_{-r}\left|\nabla\phi\right|^2\frac{\eta(t)^2}{\eta'(t)}dt\notag\\
\leq&\,4\sup_{|t|\leq{r}}\frac{\eta(t)^2}{\eta'(t)^2}\int^r_{-r}|\nabla\phi|^2\eta'(t)dt
\end{align}
in view of the Cauchy-Schwarz inequality. Thus
\[
\int_{B_r}\phi^2dV
= \int^{2\pi}_0\int^r_{-r}\phi^2\eta'(t)dtd\theta \leq 4\sup_{|t|\leq{r}}\frac{\eta(t)^2}{\eta'(t)^2}\int_{B_r}|\nabla\phi|^2dV,
\]
from which the assertion immediately follows.
\end{proof}

\begin{example}
Given $\alpha>0$, take $\eta$ such that
\[
\eta(t)=\begin{cases}
(-t)^{-\alpha},\ \ \ &t<-1,\\
2t^\alpha,\ \ \  &t>1.
\end{cases}
\]
Then
\begin{itemize}
\item[$(1)$]
$\Lambda_*\sim\alpha^2/4$ as $\alpha\rightarrow+\infty$.

\item[$(2)$]
$M$ is parabolic if and only if\/ $0<\alpha\leq2$.
\end{itemize}
\end{example}

\begin{proof}
(1) By Proposition \ref{prop:example}, we have
\[
\Lambda_*=\liminf_{r\rightarrow +\infty} \{ r^2 \lambda_1(B_r)\}\geq\frac{\alpha^2}{4}.
\]
On the other hand, since
\[
|B_r|=2\pi\int^r_{-r}\eta'(t)dt=2\pi(\eta(r)-\eta(-r))=4\pi{r}^\alpha-2\pi{r}^{-\alpha},\ \ \ \forall\,r\gg1,
\]
we see that
\[
\nu_*=\liminf_{r\rightarrow +\infty} \frac{\log |B_r|}{\log r}=\alpha.
\]
This together with Corollary \ref{cor:Upper} gives
\[
\frac{\alpha^2}{4}\leq\Lambda_*\leq\lambda_{\alpha}=j^2_{\alpha/2-1}\sim\frac{\alpha^2}{4}
\]
i.e., $\Lambda_*\sim\alpha^2/4$ as $\alpha\rightarrow+\infty$. In particular, $M$ is hyperbolic provided $\alpha\gg1$, in view of Corollary \ref{cor:hyperbolic}. 

(2) follows immediately from Proposition \ref{prop:example_hyperbolic}.
\end{proof}

\begin{example}\label{eg:exponential}
Given $\alpha>0$, take $\eta$ such that
\begin{equation}\label{eq:eta_exp}
\eta'(t)=e^{-\alpha|t|},\ \ \ |t|>1.
\end{equation}
Then
\begin{equation}\label{eq:exp_example_2}
\lambda_1(B_r)\gtrsim{e^{-\alpha{r}}},
\end{equation}
and
\[
\liminf_{r\rightarrow+\infty}\frac{-\log\lambda_1(B_r)}{r}=\alpha=\liminf_{r\rightarrow+\infty}\frac{-\log|M\setminus B_r|}{r},
\]
i.e., the estimate in Proposition \ref{prop:volume_exponential} is sharp.
\end{example}

\begin{proof}
First of all, since
\[
|M\setminus B_r|=4\pi\int^\infty_re^{-\alpha{t}}dt\asymp{e^{-\alpha{r}}},\ \ \ r\gg1,
\]
we have $\liminf_{r\rightarrow+\infty}\frac{-\log|M\setminus B_r|}{r}=\alpha$, which implies
\[
\liminf_{r\rightarrow+\infty}\frac{-\log\lambda_1(B_r)}{r}\geq\alpha,
\]
in view of Proposition \ref{prop:volume_exponential}.

Next, we shall use the following Hardy-type inequality (cf.  Opic-Kufner \cite{OpicKufner}, pp. 100--103)
\begin{equation}\label{eq:Hardy_exp}
\int^r_{-r}\phi(t)^2\eta'(t)dt\lesssim{e^{\alpha{r}}}\int^r_{-r}\phi'(t)^2\eta'(t)dt,\ \ \ \forall\,\phi\in{C^\infty_0((-r,r))},
\end{equation}
where the implicit constant is independent of $r$. For reader's convenience, we include here a rather short proof for this special case. Since $\int^{+\infty}_{-\infty}\eta'(t)dt$ is finite in view of \eqref{eq:eta_exp}, we have
\begin{equation}\label{eq:Hardy_exp_LHS}
\int^r_{-r}\phi(t)^2\eta'(t)dt\leq\sup_{-r<{t}<{r}}\phi(t)^2\int^r_{-r}\eta'(t)dt\lesssim\sup_{-r<{t}<{r}}\phi(t)^2.
\end{equation}
On the other hand, by setting $|\phi(t_0)|=\sup_{-r<{t}<{r}}|\phi(t)|$, we have
\[
\int^r_{-r}|\phi'(t)|dt\geq\int^{t_0}_{-r}|\phi'(t)|dt\geq\left|\int^{t_0}_{-r}\phi'(t)dt\right|=|\phi(t_0)|=\sup_{-r<{t}<{r}}|\phi(t)|.
\]
This together with the Cauchy-Schwarz inequality yield
\begin{equation}\label{eq:Hardy_exp_RHS}
\sup_{-r<{t}<{r}}\phi(t)^2\leq\left(\int^r_{-r}\phi'(t)^2\eta'(t)dt\right)\left(\int^r_{-r}\frac{1}{\eta'(t)}dt\right)\lesssim{e^{\alpha{r}}}\int^r_{-r}\phi'(t)^2\eta'(t)dt.
\end{equation}
By \eqref{eq:Hardy_exp_LHS} and \eqref{eq:Hardy_exp_RHS}, we obtain \eqref{eq:Hardy_exp}, which in turn gives \eqref{eq:exp_example_2}, i.e.,
\[
\liminf_{r\rightarrow+\infty}\frac{-\log(\lambda_1(B_r))}{r}\leq\alpha.\qedhere
\]
\end{proof}

\begin{remark}
By Proposition \ref{prop:example}, we only obtain a weaker conclusion
\[
\lambda_1(B_r)\geq\frac{1}{4}\frac{\eta'(r)^2}{\eta(r)^2}\gtrsim{e^{-2\alpha{r}}}.
\]
\end{remark}

\begin{example}\label{eg:counter_example}
Take $\eta$ such that
\[
\eta'(t)=\begin{cases}
n^2,\ \ \ &\,2^n<t<2^n+1,\\
e^t,\ \ \ &\,2^n+2<t<2^{n+1}-1\\
\end{cases}
\]
and
\[
\eta'(t)\geq{n^2},\ \ \ 2^n-1\leq{t}\leq{2^n+2}
\]
for $n=2,3,\cdots$. Then
\[
\lambda_1(B_r)\lesssim{e^{-r/2}(\log{r})^2},\ \ \ \forall\,r\gg1
\]
and $M$ is hyperbolic.
\end{example}

\begin{proof}
Define $h$ as Proposition \ref{prop:example_hyperbolic}. Since $\sup{h}=\int^{+\infty}_0\frac{dt}{\eta'(t)}<+\infty$, we infer from Proposition \ref{prop:example_hyperbolic} that $M$ is hyperbolic. Use the test function,
\[
\phi(t):=\begin{cases}
t,\ \ \ &\,0\leq{t}\leq{1},\\
1,\ \ \ &\,1\leq{t}\leq{r-1},\\
r-t,\ \ \ &\,r-1\leq{t}\leq{r},\\
0,\ \ \ &\,\text{otherwise},
\end{cases}
\]
we obtain
\begin{equation}\label{eq:counter_example_upper}
\lambda_1(B_r)\leq\frac{\int^r_{-r}\phi'(t)^2\eta'(t)dt}{\int^r_{-r}\phi(t)^2\eta'(t)dt}\leq\frac{\eta(r)-\eta(r-1)+\eta(1)-\eta(0)}{\eta(r-1)-\eta(1)}.
\end{equation}
For $r=2^n+1$, where $n\in\mathbb{Z}^+$ and $n\geq2$, we have $\eta(r)-\eta(r-1)=n^2$ and
\[
\eta(r-1)=\eta(2^n)\geq\int^{2^n-1}_{2^{n-1}+2}e^tdt\asymp{e^{2^n}},
\]
so that \eqref{eq:counter_example_upper} gives
\[
\lambda_1(B_{2^n+1})\lesssim{e^{-2^n}n^2}.
\]
In general, for $r\gg1$, take $n\in\mathbb{Z}^+$ such that $2^n+1\leq{r}<2^{n+1}+1$. Then
\[
\lambda_1(B_r)\leq\lambda_1(B_{2^n+1})\lesssim{e^{-2^n}n^2}\lesssim{e^{-r/2}(\log{r})^2}.\qedhere
\]
\end{proof}

\begin{remark}
We have
\[
\Lambda^*:=\limsup_{r\rightarrow+\infty}r^2\lambda_1(B_r)=0
\]
and
\[
\widetilde{\Lambda}_*=\liminf_{r\rightarrow+\infty}\frac{-\log\lambda_1(B_r)}{r}\geq\frac{1}{2}>0.
\]
On the other hand, $M$ is hyperbolic with an infinite volume.
\end{remark}

\begin{example}

Let $\mu$ be a positive, smooth and decreasing function on $[1,+\infty)$ satisfying
\begin{itemize}
\item[$(1)$]
$\lim_{t\rightarrow+\infty}\mu(t)=0$,

\item[$(2)$]
$\int^{+\infty}_1\mu(s)ds=+\infty$,

\item[$(3)$]
$t\mu(t)$ is increasing on $[c,+\infty)$ for some $c\gg1$.
\end{itemize}
Take $\eta$ such that
\[
\eta(t)=\begin{cases}
e^{-\int^{-t}_1\mu(s)ds},\ \ \ &t<-1,\\
2e^{\int^t_1\mu(s)ds},\ \ \ &t>1.
\end{cases}
\]
Then
\begin{equation}\label{eq:lambda_decay_mu}
\lambda_1(B_r)\asymp\mu(r)^2.
\end{equation}
\end{example}

\begin{proof}
Note that $\eta'(t)/\eta(t)=\mu(-t)$ for $t<-1$ and $\eta'(t)/\eta(t)=\mu(t)$ for $t>1$. Thus it follows from Proposition \ref{prop:example} that
\begin{equation}\label{eq:mu_behaviour_lower}
\lambda_1(B_r)\geq\frac{1}{4}\inf_{|t|\leq{r}}\frac{\eta'(t)^2}{\eta(t)^2}=\frac{\mu(r)^2}{4}.
\end{equation}

On the other hand, we have $r\mu(r)\geq{c\mu(c)}>0$ for $r\geq{c}\gg1$ in view of the condition (3). Thus we may take $0<\varepsilon\leq{c\mu(c)}/2$ so that
\begin{equation}\label{eq:r_epsilon}
r_\varepsilon:=r-\varepsilon\mu(r)^{-1}=r\left(1-\varepsilon{r}^{-1}\mu(r)^{-1}\right)\geq\frac{r}{2},\ \ \ \forall\,r\geq{c}.
\end{equation}
Set $I_r:=(-r,-r_\varepsilon)$. Since $\eta''(t)=-\mu'(-t)\eta(t)+\mu(-t)\eta'(t)\geq0$, i.e., $\eta'(t)$ is increasing on $(-\infty,-1]$, it follows that
\begin{align*}
\lambda_1(B_r)
\leq&\, \lambda_1\left(\{(t,\theta)\in{M}:-r\leq{t}\leq{-r_\varepsilon}\}\right)\\
\leq&\, \inf_{\phi\in{C^\infty_0(I_r)}}\left\{\frac{\int_{I_r}\phi'(t)^2\eta'(t)dt}{\int_{I_r}\phi(t)^2\eta'(t)dt}\right\}\\
\leq&\, \inf_{\phi\in{C^\infty_0(I_r)}}\left\{\frac{\int_{I_r}\phi'(t)^2dt}{\int_{I_r}\phi(t)^2dt}\right\}\cdot\frac{\eta'(-r_\varepsilon)}{\eta'(-r)}\\
=&\, \lambda_1(I_r)\cdot\frac{\eta'(-r_\varepsilon)}{\eta'(-r)}.
\end{align*}
Since $\lambda_1(I_r)\lesssim|I_r|^{-2}\asymp\mu(r)^2$, we obtain
\begin{equation}\label{eq:mu_behaviour_upper}
\lambda_1(B_r)\lesssim\mu(r)^2\cdot\frac{\eta'(-r_\varepsilon)}{\eta'(-r)}.
\end{equation}
We have
\[
\frac{\eta'(-r_\varepsilon)}{\eta'(-r)}=\frac{\mu(r_\varepsilon)}{\mu(r)}\exp\left(\int^r_{r_\varepsilon}\mu(s)ds\right)\leq\frac{\mu(r_\varepsilon)}{\mu(r)}\exp\left(\varepsilon\frac{\mu(r_\varepsilon)}{\mu(r)}\right),
\]
for $\mu$ is decreasing and $r-r_\varepsilon=\varepsilon\mu(r)^{-1}$. By condition (3) and \eqref{eq:r_epsilon}, we have
\[
\frac{\mu(r_\varepsilon)}{\mu(r)}\leq\frac{r}{r_\varepsilon}\leq2.
\]
Thus $\frac{\eta'(-r_\varepsilon)}{\eta'(-r)}=O(1)$ as $r\rightarrow+\infty$. This together with \eqref{eq:mu_behaviour_lower} and \eqref{eq:mu_behaviour_upper} give \eqref{eq:lambda_decay_mu}.
\end{proof}

Particular choices of $\mu$ give the following
\begin{itemize}
\item[$(1)$]
For $\mu(t)=t^{-1}(\log{t})^\beta$ with $\beta>0$, $\lambda_1(B_r)\asymp{r^{-2}}(\log{r})^{2\beta}$.

\item[$(2)$]
For $\mu(t)=t^{-\alpha}$ with $0<\alpha<1$, $\lambda_1(B_r)\asymp{r^{-2\alpha}}$.

\item[$(3)$] For $\mu(t)=(\log(t+1))^{-\gamma}$ with $\gamma>0$, $\lambda_1(B_r)\asymp(\log{r})^{-2\gamma}$.
\end{itemize}
\noindent In all three cases, we have
\[
\Lambda_*=\liminf_{r\rightarrow+\infty}\left\{r^2\lambda_1(B_r)\right\}=+\infty.
\]
Thus these Riemannian manifolds $(M,g)$ are hyperbolic in view of Corollary \ref{cor:hyperbolic}.

\end{document}